\documentclass[reqno,11pt]{amsart}
\usepackage[margin=1in]{geometry}

\usepackage{amsthm, amsmath, amssymb, bm, mathtools}
\usepackage{tikz}

\usepackage[utf8]{inputenc}
\usepackage[T1]{fontenc}

\usepackage[textsize=small,backgroundcolor=orange!20]{todonotes}

\usepackage[hidelinks]{hyperref}
\usepackage{url}

\usepackage[noabbrev,capitalize]{cleveref}
\crefname{equation}{}{}

\usepackage{url}
\usepackage[color,notcite,final]{showkeys} %add in 'final' into parameter to remove showkeys

\colorlet{refkey}{orange!20}
\colorlet{labelkey}{blue!60}

\newtheorem{theorem}{Theorem}[section]

\newtheorem{lemma}[theorem]{Lemma}

\newtheorem{corollary}[theorem]{Corollary}

\theoremstyle{definition}

\newtheorem{question}[theorem]{Question}

\theoremstyle{remark}

\newcommand{\FF}{\mathbb{F}}

\newcommand{\oFF}{\overline{\FF}}

\title{Some remarks on the Zarankiewicz problem} 

\author{David Conlon}
\address{Department of Mathematics, California Institute of Technology, Pasadena, CA 91125, USA}
\email{dconlon@caltech.edu}

\begin{document}

\begin{abstract}
The Zarankiewicz problem asks for an estimate on $z(m, n; s, t)$, the largest number of $1$'s in an $m \times n$ matrix with all entries $0$ or $1$ containing no $s \times t$ submatrix consisting entirely of $1$'s. We show that a classical upper bound for $z(m, n; s, t)$ due to K\H{o}v\'ari, S\'os and Tur\'an is tight up to the constant for a broad range of parameters. The proof relies on a new quantitative variant of the random algebraic method.
\end{abstract}

\maketitle

\section{Introduction}

The classical Zarankiewicz problem~\cite{Z51} asks for an estimate on $z(m, n; s, t)$, the maximum number of edges in a bipartite graph $G = (U, V; E)$ with $|U| = m$ and $|V| = n$ containing no copy of $K_{s,t}$. The order here is important, in that the set of size $s$ in $K_{s,t}$ must be embedded in $U$, while the set of size $t$ must be in $V$. Equivalently, $z(m, n; s, t)$ is the largest number of $1$'s in an $m \times n$ matrix with all entries $0$ or $1$ containing no $s \times t$ submatrix consisting entirely of $1$'s.

A result of K\H{o}v\'ari, S\'os and Tur\'an~\cite{KST54} says that
\[z(m, n; s, t) = O(m n^{1-1/s} + n),\]
where the implied constant depends only on $s$ and $t$ (a convention that we adopt throughout). If we swap the roles of $m$ and $n$ and of $s$ and $t$, we also obtain the bound
\[z(m, n; s, t) = O(n m^{1-1/t} + m).\]
Assuming that $s \leq t$, the point where the second bound becomes better than the first is when $m$ is on the order of $n^{t/s}$. It turns out that this crossover point is critical to the problem, in that a lower bound for $z(n^{t/s}, n; s, t)$ which matches the upper bound up to a constant implies, by a simple sampling argument, a lower bound for $z(m, n; s, t)$ which is tight up to the constant for all $m$.

These observations give rise to the following attractive question. 

\begin{question} \label{qn:main}
Is it the case that for any fixed $s$ and $t$ with $2 \leq s \leq t$ and any $m \leq n^{t/s}$,
\[z(m, n; s, t) = \Omega(m n^{1-1/s})?\]
\end{question}

In light of this question, our current state of knowledge about lower bounds for the Zarankiewicz problem seems rather weak. The classic result in the area is due to Koll\'ar, R\'onyai and Szab\'o~\cite{KRSz}, who showed\footnote{Their result actually gives something considerably stronger, namely, a graph with $n$ vertices and $\Omega(n^{2-1/s})$ edges containing no copy of $K_{s,t}$ with no regard for how it is oriented in the graph. However, the corollary stated here is all that is relevant to our discussion.} that for any $s$ there exists $t$ such that
\[z(n, n; s, t) = \Omega(n^{2-1/s}).\]
In their work, it suffices to take $t \geq s! + 1$, a bound that was subsequently improved by Alon, R\'onyai and Szab\'o~\cite{ARSz} to $t \geq (s-1)! + 1$. Unfortunately, this only furnishes a complete answer to Question~\ref{qn:main} in two cases, when $s = t = 2$ and when $s = t = 3$ (and both of these cases were fully resolved much earlier~\cite{Bro66, Erd38}). However, a result of Alon, Mellinger, Mubayi and Verstra\"ete~\cite{AMMV12} shows that Question~\ref{qn:main} also has a positive answer for $s = 2$ and arbitrary $t$.

\begin{theorem}[Alon--Mellinger--Mubayi--Verstra\"ete~\cite{AMMV12}] \label{thm:2case}
For any fixed $t \geq 2$ and any $m \leq n^{t/2}$,
\[z(m, n; 2, t) = \Omega(m n^{1/2}).\]
\end{theorem}

In recent years, alternative proofs of the Koll\'ar--R\'onyai--Szab\'o theorem, though with weaker control on $t$, were found by Blagojevi\'c, Bukh and Karasev~\cite{BBK} and by Bukh~\cite{B15} using constructions where adjacency is determined by a randomly chosen algebraic variety. The first traces of this random algebraic method go back some way, to work of Matou\v sek~\cite{Mat} in discrepancy theory, but it is the variant originating with Bukh~\cite{B15}, and developed further by the author~\cite{Con}, that has proved most useful. For instance, it has led to considerable progress~\cite{BC, CJL, Jan, JJM, JMY, JQ, KKL18} on the celebrated rational exponents conjecture of Erd\H{o}s and Simonovits~\cite{E81}, amongst other applications~\cite{BG, CT}. Our main result, a general lower bound for $z(m, n; s, t)$ valid over a broad range of $m$, is another application of the random algebraic method, though in a new, arguably simpler, form that returns quantitative estimates not at present available through the application of Bukh's method.

\begin{theorem} \label{thm:main}
For any fixed $2 \leq s \leq t$ and any $m \leq n^{t^{1/(s-1)}/s(s-1)}$,
\[z(m, n; s, t) = \Omega(m n^{1-1/s}).\]
\end{theorem}

This may be seen as partial progress on Question~\ref{qn:main}, even if it leaves considerable room for improvement (except in the $s = 2$ case, where it agrees with Theorem~\ref{thm:2case}). On the other hand, we note that the only previous results dealing with lower bounds when $s \geq 3$ and $m = \omega(n)$ are a result of Matou\v sek~\cite{Mat} and a result of Alon, R\'onyai and Szab\'o~\cite{ARSz}, which subsumes that of Matou\v sek, saying that the conclusion of Theorem~\ref{thm:main} holds for any fixed $s \geq 2$ and $t \geq s! + 1$ and any $m \leq n^{1+1/s}$. Our result only begins to match the Alon--R\'onyai--Szab\'o result when $t$ is roughly $s^{2s}$, but rapidly improves on it for larger $t$.

\section{Random polynomials and varieties}

Let $q$ be a prime power and let $\mathbb{F}_q$ be the finite field of order $q$. We will consider polynomials in $t$ variables over $\mathbb{F}_q$, writing any such polynomial as $f(X)$, where $X = (X_1, \dots, X_t)$. We let $\mathcal{P}_d$ be the set of polynomials in $X$ of degree at most $d$, that is, the set of linear combinations over $\mathbb{F}_q$ of monomials of the form $X_1^{a_1} \cdots X_t^{a_t}$ with $\sum_{i=1}^t a_i \leq d$. By a random polynomial, we just mean a polynomial chosen uniformly from the set $\mathcal{P}_d$. One may produce such a random polynomial by taking the coefficients of the monomials above to be independent random elements of $\mathbb{F}_q$.

The next lemma estimates the probability that a randomly chosen polynomial from $\mathcal{P}_d$ passes through each of $m$ distinct points. This is very similar to a result of Bukh~\cite{B15}, but there is a crucial difference, in that we are interested in whether our random polynomial passes through points in $\oFF_q^t$, where $\oFF_q$ is the algebraic closure of $\FF_q$, and not just $\FF_q^t$.

\begin{lemma} \label{prob}
Suppose that $q > \binom{m}{2}$ and $d \geq m - 1$. If $f$ is a random $t$-variate polynomial of degree $d$ over $\mathbb{F}_q$ and $x_1, \dots, x_m$ are $m$ distinct points in $\oFF_{q}^t$, then
$$\mathbb{P}[f(x_i) = 0 \mbox{ for all } i = 1, \dots, m] \leq 1/q^{m}.$$ 
\end{lemma}

\begin{proof}
Let $x_i = (x_{i,1}, \dots, x_{i,t})$ for each $i = 1, \dots, m$. We choose elements $a_2, \dots, a_t  \in \mathbb{F}_q$ such that $x_{i,1} + \sum_{j=2}^t a_j x_{i, j}$ is distinct for all $i = 1, \dots, m$. To see that this is possible, note that there are exactly $\binom{m}{2}$ equations 
\[x_{i,1} + \sum_{j=2}^t a_j x_{i, j} = x_{i',1} + \sum_{j=2}^t a_j x_{i', j},\] 
each with at most $q^{t-2}$ solutions $(a_2, \dots, a_t)$. Therefore, since the total number of choices for $(a_2, \dots, a_t)$ is $q^{t-1}$ and $q^{t-1} > q^{t-2} \binom{m}{2}$, we can make an appropriate choice.

We now consider $\mathcal{P}'_d$, the set of polynomials of degree at most $d$ in $Z$, where $Z_1 = X_1 + \sum_{j=2}^t a_j X_j$ and $Z_j = X_j$ for all $2 \leq j \leq t$. Since this change of variables is an invertible linear map, $\mathcal{P}'_d$ is identical to $\mathcal{P}_d$. It will therefore suffice to show that a randomly chosen polynomial from $\mathcal{P}'_d$ passes through all of the points $z_1, \dots, z_m$ corresponding to $x_1, \dots, x_m$ with probability at most $q^{-m}$. For this, we will need the fact that, by our choice above, $z_{i,1} \neq z_{i',1}$ for any $1 \leq i < i' \leq m$.

For any $f$ in $\mathcal{P}'_d$, we may write $f = g + h$, where $h$ contains all monomials of the form $Z_1^j$ for $j = 0, 1, \dots, m-1$ and $g$ contains all other monomials. For any fixed choice of $g$, there is exactly one choice of $h$ with coefficients in $\oFF_{q}$ such that $f(z_i) = 0$ for all $i = 1, \dots, m$, namely, the unique polynomial of degree at most $m-1$ which takes the value $-g(z_i)$ at $z_{i,1}$ for all $i = 1, 2, \dots, m$, where uniqueness follows from the fact that the $z_{i,1}$ are distinct. Therefore, the number of choices for $h$ with coefficients in $\mathbb{F}_q$ is either $0$ or $1$. Since this is out of a total of $q^m$ possibilities, we see that the probability $f$ passes through all of the $z_i$ is at most $q^{-m}$, as required.
\end{proof}

Despite yielding quantitative results that were unavailable to earlier versions of the random algebraic method, our method relies on rather less input from algebraic geometry. Recall that a variety over an algebraically closed field $\oFF$ is a set of the form
\[W = \{x \in \oFF^t : f_1(x) = \dots = f_s(x) = 0\}\]
for some collection of polynomials $f_1, \dots, f_s : \oFF^t \rightarrow \oFF$. The variety is irreducible if it cannot be written as the union of two proper subvarieties. The dimension $\dim W$ of $W$ is then the maximum integer $d$ such that there exists a chain of irreducible subvarieties of $W$ of the form
\[\emptyset \subsetneq \{p\} \subsetneq W_1 \subsetneq W_2 \subsetneq \dots \subsetneq W_d \subset W,\]
where $p$ is a point. The following three standard lemmas about varieties will suffice for our purposes.

%Reference: https://math.stackexchange.com/questions/1713906/how-is-finiteness-of-solutions-zero-dimensionality-related-to-krulls-dimensio
%Proposition 3.7.1 of Kreuzer-Robbiano - Computational Commutative Algebra 1

\begin{lemma} \label{infinite}
Every variety $W$ over an algebraically closed field $\oFF$ with $\dim W \geq 1$ has infinitely many points.
\end{lemma}

%Reference: https://math.stackexchange.com/questions/1718137/dimension-and-intersection-of-algebraic-varieties

\begin{lemma} \label{dimdrop}
Suppose that $W$ is an irreducible variety over an algebraically closed field $\oFF$. Then, for any polynomial $g: \oFF^t \rightarrow \oFF$, $W \subseteq \{x : g(x) = 0\}$ or $W \cap \{x : g(x) = 0\}$ is a variety of dimension less than $\dim W$.
\end{lemma}

%Reference: https://terrytao.wordpress.com/2011/03/23/bezouts-inequality/
%https://math.stackexchange.com/questions/1115337/number-of-connected-components-of-a-real-variety
%Section 11 of https://www2.math.ethz.ch/education/bachelor/lectures/fs2015/math/alg_geom/lecturenotes
%Proposition 13 of https://www.ams.org/journals/jams/2013-26-04/S0894-0347-2013-00764-X/S0894-0347-2013-00764-X.pdf
%https://math.stackexchange.com/questions/1348589/basic-question-regarding-degrees-of-algebraic-sets
%Page 10 of Fulton's book Introduction to intersection theory in algebraic geometry

\begin{lemma}[B\'ezout's theorem~\cite{F84}] \label{bezout}
If, for a collection of polynomials $f_1, \dots, f_t : \oFF^t \rightarrow \oFF$, the variety
\[W = \{x \in \oFF^t : f_1(x) = \dots = f_t(x) = 0\}\]
has $\dim W = 0$, then 
\[|W| \leq \prod_{i=1}^t \deg(f_i).\]
Moreover, for a collection of polynomials $f_1, \dots, f_s : \oFF^t \rightarrow \oFF$, the variety
\[W = \{x \in \oFF^t : f_1(x) = \dots = f_s(x) = 0\}\]
has at most $\prod_{i=1}^s \deg(f_i)$ irreducible components.
\end{lemma}

\section{Proof of Theorem~\ref{thm:main}}

Fix $d = \lceil t^{1/(s-1)}\rceil - 1$ and $\ell = \lfloor \frac{1}{2d} q^{(d+1)/(s-1)} \rfloor$. Consider the bipartite graph between sets $U$ and $V$, where $V$ may be viewed as a copy of $\FF_q^s$ for some prime power $q$ and $U$ has order $\ell$, each vertex $u_i$ of which is associated to an $(s-1)$-variate polynomial $f_i$ of degree at most $d$ with coefficients in $\FF_q$. Each $u_i$ is then joined to the set of points 
$$S_i = \{(x_1, \dots, x_{s-1}, f_i(x_1, \dots, x_{s-1})) : x_1, \dots, x_{s-1} \in \FF_q\}$$
in $V$. Note that, for any $1 \leq j \leq s$ and $1 \leq i_1 < \dots < i_j \leq \ell$,
\[S_{i_1} \cap \cdots \cap S_{i_j} = \{(x_1, \dots, x_s) : x_s = f_{i_1}(x_1, \dots, x_{s-1}) = \dots = f_{i_j}(x_1, \dots, x_{s-1}) \}.\]
This intersection therefore has the same size as $T_{i_1, i_2} \cap \dots \cap T_{i_1, i_j}$, where 
\[T_{i, i'} = \{(x_1, \dots, x_{s-1}) : (f_i - f_{i'})(x_1, \dots, x_{s-1}) = 0\}.\]
Our aim now is to show that there is a choice of $f_i$ for $i = 1, \dots, \ell$ such that, for any $1 \leq j \leq s$ and $1 \leq i_1 < \dots < i_j \leq \ell$, the intersection $T_{i_1, i_2} \cap \dots \cap T_{i_1, i_j}$, considered as a variety over $\oFF_q$, has dimension at most $s-j$. To do this, we will pick the $f_i$ in sequence and show, by induction, that for every $1 \leq k \leq \ell$, there exist $f_1, \dots, f_k$ such that $T_{i_1, i_2} \cap \dots \cap T_{i_1, i_j}$ has dimension at most $s - j$ for any $1 \leq j \leq s$ and $1 \leq i_1 < \dots < i_j \leq k$.

To begin the induction, we let $f_1$ be any $(s-1)$-variate polynomial  of degree $d$. In this case, the condition that the intersection $T_{i_1, i_2} \cap \dots \cap T_{i_1, i_j}$ have dimension at most $s-j$ for all $1 \leq j \leq s$ and $1 \leq i_1 < \dots < i_j \leq k$ is degenerate, but can be meaningfully replaced by the observation that the set of all $(x_1, \dots, x_{s-1})$, corresponding to the trivial intersection, equals $\oFF_q^{s-1}$, which has dimension $s-1$, as required.

Suppose now that $f_1, \dots, f_{k-1}$ have been chosen consistent with the induction hypothesis. We would like to pick $f_k$ so that for any $1 \leq j \leq s$ and $1 \leq i_1 < \dots < i_{j-1} < k$, the intersection $T_{i_1, i_2} \cap \dots \cap T_{i_1, i_{j-1}} \cap T_{i_1, k}$ has dimension at most $s-j$. For now, fix $1 \leq j \leq s$ and $1 \leq i_1 < \dots < i_{j-1} < k$ and note, by the induction hypothesis, that $T_{i_1, i_2} \cap \dots \cap T_{i_1, i_{j-1}}$ has dimension at most $s - j + 1$. 

Split the variety $T_{i_1, i_2} \cap \dots \cap T_{i_1, i_{j-1}}$ into irreducible components $W_1, \dots, W_r$ and suppose that $W_a$ is a component of dimension $s - j + 1 \geq 1$. By Lemma~\ref{infinite}, $W_a$ has infinitely many points. Fix $d+1$ points $w_1, \dots, w_{d+1}$ on $W_a$. For any $(s-1)$-variate polynomial $f$, write
$$T_{i_1, f} = \{(x_1, \dots, x_{s-1}) : (f - f_{i_1})(x_1, \dots, x_{s-1}) = 0\}.$$
By Lemma~\ref{dimdrop}, we see that if $\dim W_a \cap T_{i_1, f} = \dim W_a$, then $T_{i_1, f}$ must contain all of $W_a$ and, in particular, each of $w_1, \dots, w_{d+1}$. Therefore, for a random $(s-1)$-variate polynomial $f$ of degree $d$, the probability that $W_a \cap T_{i_1, f}$ does not have dimension at most $s - j$ is at most the probability that the polynomial $f - f_{i_1}$ passes through all of $w_1, \dots, w_{d+1}$, which, by Lemma~\ref{prob}, is at most $q^{-(d+1)}$.

Since, by Lemma~\ref{bezout}, the number of irreducible components of $T_{i_1, i_2} \cap \dots \cap T_{i_1, i_{j-1}}$ is at most $d^{s-1}$, this implies that the probability $T_{i_1, i_2} \cap \dots \cap T_{i_1, i_{j-1}} \cap T_{i_1,f}$ does not have dimension at most $s - j$ is at most $d^{s-1} q^{-(d+1)}$. By taking a union bound over the at most $\ell^{s-1}$ choices for $j$ and $i_1, \dots, i_{j-1}$, we see that the probability there exists $1 \leq j \leq s$ and $1 \leq i_1 < \dots < i_{j-1} < k$ such that $T_{i_1, i_2} \cap \dots \cap T_{i_1, i_{j-1}} \cap T_{i_1,f}$ does not have dimension at most $s- j$ is at most $\ell^{s-1} d^{s-1} q^{-(d+1)} < 1$ for $q$ sufficiently large. Therefore, there exists an $(s-1)$-variate polynomial $f$ of degree at most $d$ such that $T_{i_1, i_2} \cap \dots \cap T_{i_1, i_{j-1}} \cap T_{i_1,f}$ has dimension at most $s - j$ for any $1 \leq j \leq s$ and $1 \leq i_1 < \dots < i_{j-1} < k$, so taking $f_k = f$ completes the induction.

To conclude the proof of Theorem~\ref{thm:main}, we note that for any $1 \leq i_1 < \dots < i_s \leq \ell$ the intersection $T_{i_1, i_2} \cap \dots \cap T_{i_1, i_{s}}$ has dimension zero, so, by B\'ezout's theorem, Lemma~\ref{bezout}, the number of points in the intersection is at most $d^{s-1} < t$. Therefore, for any $1 \leq i_1 < \dots < i_s \leq \ell$, the intersection $S_{i_1} \cap \dots \cap S_{i_s}$ has at most $t-1$ points, so there is no copy of $K_{s,t}$ with $s$ vertices in $U$ and $t$ vertices in $V$. Since $|U| = \ell = \Omega(q^{(d+1)/(s-1)})$, $|V| = q^s$ and $|E| = \ell q^{s-1}$, we thus have, for $m_0 \coloneqq m_0(n) = n^{(d+1)/s(s-1)} \geq n^{t^{1/(s-1)}/s(s-1)}$, that
\[z(m_0, n; s, t) = \Omega(m_0 n^{1 - 1/s})\]
when $n$ is of the form $q^s$ with $q$ a prime power. By applying Bertrand's postulate, we can easily extend this result to all $n$. Finally, for any $m \leq m_0$, we can verify that $z(m,n;s, t) = \Omega(m n^{1-1/s})$ by choosing a random subset $U'$ of $U$ of order $m$ and noting that the expected number of edges between $U'$ and $V$ is $\Omega(m n^{1 - 1/s})$. Therefore, there must exist some choice for $U'$ such that the number of edges between $U'$ and $V$ is $\Omega(m n^{1 - 1/s})$, demonstrating the required lower bound. 

\vspace{3mm}
\noindent
{\bf Remark.}
It is worth remarking that the method becomes a little simpler if we instead consider
$$S_i = \{x \in \FF_q^s : f_i(x) = 0\}$$
for random $s$-variate polynomials $f_i$ of degree at most $d$ with coefficients in $\FF_q$. There is, however, a small tradeoff in the bound, in that it only allows us to determine a tight bound for $z(m, n; s,t)$ for $m \leq n^{t^{1/s}/s(s-1)}$ rather than $m \leq n^{t^{1/(s-1)}/s(s-1)}$. The difference is quite similar to the difference between the norm graphs of~\cite{KRSz} and the projective norm graphs of~\cite{ARSz}.

\section{Concluding remarks}

The result of Alon, Mellinger, Mubayi and Verstra\"ete, Theorem~\ref{thm:2case}, implies that $z(n, n^{t/2}; t, t) \geq z(n, n^{t/2}; t, 2) = \Omega(n^{(t+1)/2})$. By replacing $n$ with $n^{2/t}$, we see that $z(n^{2/t}, n; t, t) = \Omega(n^{1 + 1/t})$, which agrees with the upper bound up to a constant. By the usual sampling argument, this yields the following corollary.

\begin{corollary} \label{cor:2case}
For any fixed $t \geq 2$ and any $m \leq n^{2/t}$,
\[z(m, n; t, t) = \Omega(m n^{1 - 1/t}).\]
\end{corollary}

That is, there is an asymptotically tight estimate for the Zarankiewicz problem when $s = t$ and one side is much larger than the other (and, crucially, in the non-trivial range where $m = \omega(n^{1/t})$ and $z(m, n; t, t) = \omega(n)$). A tight lower bound for $z(m, n; s, t)$ for all $m \leq n^{t/s}$ would similarly provide a tight lower bound for $z(m, n; t, t)$ for all $m \leq n^{s/t}$. However, we believe this will be difficult, or perhaps even impossible, to achieve for $s \geq 3$. In particular, the first open case of Question~\ref{qn:main}, where $s = 3$ and $t = 4$, seems hard.

\vspace{3mm}
\noindent
{\bf Acknowledgements.} I would like to thank Cosmin Pohoata for helpful discussions. I am also grateful to Dhruv Mubayi for drawing my attention to his work with Alon, Mellinger and Verstra\"ete~\cite{AMMV12}.


\begin{thebibliography}{}

\bibitem{AMMV12}
N. Alon, K. E. Mellinger, D. Mubayi and J. Verstra\"ete, The de Bruijn--Erd\H{o}s theorem for hypergraphs, {\it Des. Codes Cryptogr.} {\bf 65} (2012), 233--245.

\bibitem{ARSz}
N. Alon, L. R\'onyai and T. Szab\'o, Norm-graphs: variations and applications, {\it J. Combin. Theory Ser. B} {\bf 76} (1999), 280--290.

\bibitem{BBK}
P. V. M. Blagojevi\'c, B. Bukh and R. Karasev, Tur\'an numbers for $K_{s,t}$-free graphs: topological obstructions and algebraic constructions, {\it Israel J. Math.} {\bf 197} (2013), 199--214.

\bibitem{Bro66}
W. G. Brown, On graphs that do not contain a Thomsen graph, {\it Canad. Math. Bull.} {\bf 9} (1966), 281--285.

\bibitem{B15}
B. Bukh, Random algebraic construction of extremal graphs, {\it Bull. Lond. Math. Soc.} {\bf 47} (2015), 939--945.

\bibitem{BC}
B. Bukh and D. Conlon, Rational exponents in extremal graph theory, {\it J. Eur. Math. Soc.} {\bf 20} (2018), 1747--1757.

\bibitem{BG}
B. Bukh and X. Goaoc, Shatter functions with polynomial growth rates, {\it SIAM J. Discrete Math.} {\bf 33} (2019), 784--794.

\bibitem{Con}
D. Conlon, Graphs with few paths of prescribed length between any two vertices, {\it Bull. Lond. Math. Soc.} {\bf 51} (2019), 1015--1021.

\bibitem{CJL}
D. Conlon, O. Janzer and J. Lee, More on the extremal number of subdivisions, to appear in {\it Combinatorica}.

\bibitem{CT}
D. Conlon and M. Tyomkyn, Repeated patterns in proper colourings, preprint available at arXiv:2002.00921 [math.CO].

\bibitem{Erd38}
P. Erd\H{o}s, On sequences of integers no one of which divides the product of two others and on some related problems, {\it Mitt. Forsch.-Inst. Math. Mech. Univ. Tomsk} {\bf 2} (1938), 74--82.

\bibitem{E81} 
P. Erd\H{o}s, On the combinatorial problems which I would most like to see solved, {\it Combinatorica} {\bf 1} (1981), 25--42.

\bibitem{F84}
W. Fulton, Introduction to intersection theory in algebraic geometry, CBMS Regional Conference Series in Mathematics, 54, American Mathematical Society, Providence, RI, 1984.

\bibitem{Jan}
O. Janzer, The extremal number of the subdivisions of the complete bipartite graph, {\it SIAM J. Discrete Math.} {\bf 34} (2020), 241--250.

\bibitem{JJM}
T. Jiang, Z. Jiang and J. Ma, Negligible obstructions and Tur\'an exponents, preprint available at arXiv:2007.02975 [math.CO].

\bibitem{JMY}
T. Jiang, J. Ma and L. Yepremyan, On Tur\'an exponents of bipartite graphs, preprint available at arXiv: 1806.02838 [math.CO].

\bibitem{JQ}
T. Jiang and Y. Qiu, Many Tur\'an exponents via subdivisions, preprint available at arXiv:1908.02385 [math.CO].

\bibitem{KKL18}
D. Y. Kang, J. Kim and H. Liu, On the rational Tur\'an exponents conjecture, {\it J. Combin. Theory Ser. B} {\bf 148} (2021), 149--172.

\bibitem{KRSz} 
J. Koll\'ar, L. R\'onyai and T. Szab\'o, Norm-graphs and bipartite Turán numbers, {\it Combinatorica} {\bf 16} (1996), 399--406.

\bibitem{KST54}
T. K\H{o}v\'ari, V. T. S\'os and P. Tur\'an, On a problem of K. Zarankiewicz, {\it Colloq. Math.} {\bf 3} (1954), 50--57.

\bibitem{Mat}
J. Matou\v sek, On discrepancy bounds via dual shatter function, {\it Mathematika} {\bf 44} (1997), 42--49.

\bibitem{Z51}
K. Zarankiewicz, Problem 101, {\it Colloq. Math.} {\bf 2}  (1951), 301.

\end{thebibliography}
\end{document}